\documentclass[11pt]{article}
\usepackage{amsmath,amsthm,amssymb}
\usepackage{url}
\usepackage{geometry}
\usepackage{graphicx,float,epstopdf}
\restylefloat{figure}

\let\OLDthebibliography\thebibliography
\renewcommand\thebibliography[1]{
  \OLDthebibliography{#1}
  \setlength{\parskip}{0pt}
  \setlength{\itemsep}{1pt plus 0.2ex}
}

\newcounter{thm}
\newtheorem{theorem}[thm]{Theorem}

\theoremstyle{definition}

\newtheorem{prop}[thm]{Proposition}
\newtheorem{remark}[thm]{Remark}

\newcommand{\C}{\mathbb{C}}

\newcommand{\Sy}{\mathcal{S}}

\newcommand{\ww}{\omega}
\newcommand{\wo}{\overline{\omega}}

\DeclareMathOperator{\Aut}{Aut}

\DeclareMathOperator{\GL}{GL}

\DeclareMathOperator{\Row}{Row}

\begin{document}
 \title{\textbf{\large{Construction of the outer automorphism of $\Sy_{6}$ via a complex Hadamard matrix}}}

 \author{\textsc{Neil I. Gillespie}\thanks{Email: neil.gillespie@bristol.ac.uk} \\
 \textit{\footnotesize{Heilbronn Institute for Mathematical Research,}}\\
    \textit{\footnotesize{University of Bristol, UK}}\\
    \textsc{Padraig \'O Cath\'ain}\thanks{Email: pocathain@wpi.edu} \\
\textit{\footnotesize{Department of Mathematical Sciences,}}\\
\textit{\footnotesize{Worcester Polytechnic Institute, MA, USA}}\\
    \textsc{Cheryl E. Praeger}\thanks{Email: cheryl.praeger@uwa.edu.au}\\
    \textit{\footnotesize{School of Mathematics and Statistics,}}\\
    \textit{\footnotesize{University of Western Australia, Australia,}}
}

 \maketitle

 \begin{center}
 \begin{abstract}
     We give a new construction of the outer automorphism of the symmetric group on six points. Our construction features a complex Hadamard matrix of order six containing third roots of unity and the algebra of split quaternions over the real numbers.
 \end{abstract}
 \end{center}

\section{Introduction}

Sylvester showed that the fifteen two-subsets of a six element set can be formed into 5 parallel classes in six different ways and that the action of $\Sy_{6}$ on these \textit{synthematic totals} is essentially different from its natural action on six points, \cite{Sylvester6}. To our knowledge this was the first construction for the outer automorphism of $\Sy_{6}$.

Miller attributes the result that for $n \neq 6$, $\Sy_{n}$ has no outer automorphisms to H\"{o}lder, and Sylvester's construction of the outer automorphism of $\Sy_{6}$ to Burnside, \cite{MillerOuter}. He also gives a by-hand construction of the outer automorphism. The papers of Janusz and Rotman, and of Ward provide easily readable accounts which are similar to Sylvester's, \cite{JanuszOuter,WardOuter}. Cameron and van Lint devoted an entire chapter (their sixth!) to the outer automorphism of $\Sy_{6}$, \cite{CameronVanLint}. They build on Sylvester's construction to construct the $5$-$(12, 6, 1)$ Witt design, the projective plane of order $4$, and the Hoffman-Singleton graph.

Via consideration of the cube in $\mathbb{R}^{3}$, Fournelle gives a heuristic for the existence of an outer automorphism of $\Sy_{6}$, and constructs it with the aid of a computer, \cite{FournelleOuter}. Howard, Millson, Snowden and Vakil give two constructions of the outer automorphism of $\Sy_{6}$, and use this to describe the invariant theory of six points in certain projective spaces, \cite{HowardOuter}.

In this note we give a construction which we believe has not previously been described, using a complex Hadamard matrix of order $6$ and a representation of the triple cover of $A_6$ over the complex numbers. This note is inspired by a construction of Marshall Hall Jr \cite{HallM12} for the outer automorphism of $M_{12}$ via a real Hadamard matrix of order $12$, and by Moorhouse's classification of the complex Hadamard matrices with doubly transitive automorphism groups, \cite{Moorhouse}. It was in the latter paper that we first became aware of the complex Hadamard matrix of order $6$ discussed in this article, where it is described as corresponding to the distance transitive triple cover of the complete bipartite graph $K_{6,6}$.

\section{Hadamard matrices}

Let $\ww$ be a primitive complex third root of unity. Then the matrix $H_6$ is \textit{complex Hadamard}.
\[
H_{6} = \left( \begin{array}{llllll}
1 & 1 & 1 & 1 & 1 & 1\\
1 & 1 & \ww & \wo & \wo & \ww \\
1 & \ww & 1 & \ww & \wo & \wo \\
1 & \wo & \ww & 1 & \ww & \wo \\
1 & \wo & \wo & \ww & 1 & \ww \\
1 & \ww & \wo & \wo & \ww & 1
\end{array}
\right) \]
This means that $H_{6}$ satisfies the identity $H_{6} H_{6}^{\dagger} = 6I_{6}$, where for an invertible complex matrix $A$, $A^{\dagger}$ is the complex conjugate transpose of $A$.
Equivalently, $H_{6}$ reaches equality in Hadamard's determinant bound. We refer the reader to \cite{deLauneyFlannery} for a comprehensive discussion of Hadamard matrices and their generalisations.

An \textit{automorphism} of a complex Hadamard matrix is a pair of monomial matrices $(P, Q)$ such that $P^{-1}HQ = H$. The set of all automorphisms of $H$ forms a group under composition. In this note we will work with the subgroup of automorphisms $(P, Q)$ where all non-zero entries are third roots of unity, we denote this group $\Aut(H)$. Consider now the projection maps $\rho_{1}(P, Q) \mapsto P$ and $\rho_{2}(P,Q) \mapsto Q$. Since $\frac{1}{\sqrt{6}}H_{6}$ is unitary, and
for any automorphism $(P, Q)$ of $H$ the identity $HQH^{-1} = P$ holds, it follows that $\rho_{1}$
and $\rho_{2}$ are conjugate representations of $\Aut(H)$. Note further that $\rho_{i}$ is a faithful
representation, since $Q = I$ forces $P = I$. Thus $\Aut(H)$ is isomorphic to a finite subgroup of monomial matrices of $\GL_{n}(\mathbb{C})$. Furthermore, if $\Aut(H)$ contains a subgroup isomorphic to $G$, then the projections $\rho_{1}$ and $\rho_{2}$ onto the first and second components of $\Aut(H)$ give two conjugate representations of $G$ by monomial matrices.

Every monomial matrix has a unique factorisation $P = DK$ where $D$ is diagonal and $K$ is a permutation matrix. The projection $\pi: P \mapsto K$ is a homomorphism for any group of monomial matrices. In general, the representation $\Aut(H)^{\rho_{1}\pi}$ is \textbf{not} linearly equivalent to the representation $\Aut(H)^{\rho_{2}\pi}$. As mentioned above, this phenomenon was first observed by Hall, who showed that the automorphism group of a Hadamard matrix of order $12$ is isomorphic to $2.M_{12}$, and that $\rho_{1}\pi$ and $\rho_{2}\pi$ realise the two inequivalent actions of $M_{12}$ on $12$ points, \cite{HallM12}. This interpretation of the outer automorphism of $M_{12}$ was also used by Elkies, Conway and Martin in their analysis of the Mathieu groupoid $M_{13}$, \cite{ConwayElkies}.

Throughout this note we use the following shorthand for monomial matrices: we list the elements of the diagonal matrix $D$, and give the cycle notation for $K$ as a permutation of the \textbf{columns} of the identity matrix (i.e. a right action).

Consider the following pairs of monomial matrices.
\begin{eqnarray*}
 \tau_1 &:=& \left([1,1,1,1,1,1] (2, 3, 4, 5, 6), \quad [1,1,1,1,1,1] (2,3,4,5,6) \right) \\
 \tau_2 &:=& \left( [1,1,\ww,\wo,\wo,\ww] (1,2) ,\quad\quad\;\; \left[1,1,\wo,\ww,\ww,\wo \right](1,2)(3,6)(4,5) \right).
\end{eqnarray*}

We define $\ast$ to be the entry-wise complex conjugation map, and consider the group $X = \langle \tau_{1}, \tau_{2}, \ast \rangle$.

\begin{prop}\label{prop1}
The group $X$ is of the form $3^{10}. \Sy_{6} . 2$.
\end{prop}

\begin{proof}
Since $\tau_{1}^{\ast} = \tau_{1}$ and $\tau_{2}^{\ast} = \tau_{2}^{-1}$, we have that $X_{0} = \langle \tau_{1}, \tau_{2}\rangle$ is normal in $X$. Hence $X = X_{0} \rtimes \langle \ast \rangle$, with $X_{0}$ of index $2$ in $X$.

The commutator $[ \tau_{2}, \ast ] = \left( [1,1,\ww,\wo,\wo,\ww] , \left[1,1,\wo,\ww,\ww,\wo \right] \right)$ consists of diagonal matrices; furthermore
$$
\tau_{2}' := [ \tau_{2}, \ast ]^{-1}\tau_{2}=((1,2),(1,2)(3,6)(4,5)),
$$
a pair of permutation matrices. Recall that $\langle s, t \mid s^{6} = t^{2} = (st)^{5} = [ t, s^{2}]^{2} = [ t, s^{3}]^{2} = 1\rangle$ is a presentation for $\Sy_{6}$ (see \cite{BealsEtAl}, for example). A computation with $t = \tau_{2}'$ and
$$
s = \tau_{1}\tau_{2}'=((1,2,3,4,5,6),(1,2,6)(3,5))
$$
shows that all the relations in this presentation hold for these elements $s, t$, and hence
$Y= \langle \tau_{1}, \tau_{2}'\rangle$ is isomorphic to a quotient of $\Sy_{6}$.
On the other hand, $Y^{\rho_{1}\pi}$ is easily seen to be isomorphic to $\Sy_{6}$, so we conclude that $Y \cong \Sy_{6}$.
Now let $N$ be the subgroup of $X$ consisting of all elements for which each component is a diagonal matrix.  Since $\tau_{1}^{\rho_{i}}$ and $\tau_{2}^{\rho_{i}}$ have determinants in $\{\pm 1\}$, every element of the projection $X_{0}^{\rho_{i}}$ has this property.
However all the elements of $N^{\rho_{i}}$ have third roots of unity along the diagonal, and so must have determinant $1$. As a result, $X_{0}^{\rho_{i}}$ is isomorphic to a subgroup of $M\rtimes \Sy_{6}$ where $M \cong 3^5$ is the group of unimodular diagonal matrices with entries from $\langle \omega \rangle$, and $\Sy_{6}$ acts as $Y^{\rho_{i}\pi}$. The only non-trivial $\Sy_{6}$-submodule of $M$ is the constant module of order $3$.

Define $n_{i+1} := [ \tau_{2}, \ast ]^{\tau_{1}^{i}}$ for each $i\geq1$. (We shift subscripts because the action of $\tau_{1}$ on $[ \tau_{2}, \ast ]$ gives elements of $N$ which have the non-initial rows of $H_{6}$ as the diagonal of the first component.) Since $[ \tau_{2}, \ast ] \in X_{0}$, we have $n_{i} \in X_{0}$ for $2 \leq i \leq 6$. Observe that
\begin{eqnarray*}
 n_{3}n_{4}^{2}n_{5}^{2} &=& \left( [1,1,1,1,\ww,\wo] ,\quad \left[1,1,1,1,\wo,\ww \right] \right)\\
 (n_{3}n_{4}^{2}n_{5}^{2})^{\tau_{2}'} &=&\left( [1,1,1,1,\ww,\wo] ,\quad \left[1,1,\ww,\wo,1,1 \right] \right).
\end{eqnarray*}
So neither of the projections $N^{\rho_{1}}$, $N^{\rho_{2}}$ are onto the constant module, and the kernel of $N^{\rho_{1}}$ is neither trivial nor the constant module. It follows that $N \cong M \times M$. Finally, we observe that monomial matrices normalise diagonal matrices, and that $X_{0}$ acts as a group of monomial matrices in each component.  It follows that $N \triangleleft X_{0}$, and that $Y$ is a complement of $N$ in $X_{0}$. Since $\ast$ acts on $N$ by inversion, $N \triangleleft X$.
\end{proof}

The group $X$ has a natural action on $6 \times 6$ matrices over $\mathbb{C}$ where $(P, Q) \in X_{0}$ acts as $H^{(P, Q)} = P^{-1}HQ$, and $\ast$ acts by complex conjugation. We compute the stabiliser of $H_{6}$ under this action. We denote this group $\Aut^{\ast}(H_{6})$ to emphasise that this is a group of semi-linear transformations in its action on the normal subgroup $N$. We require the subgroups $X_{0}$, $Y$ and $N$ defined in Proposition \ref{prop1} in the proof of the following.

\begin{prop}
The group $\Aut^{\ast}(H_{6})$ is isomorphic to the nonsplit extension $3.\Sy_{6}$, and $\Aut^{\ast}(H_{6})$ contains a $\mathbb{C}$-linear subgroup isomorphic to $3.A_{6}$.
\end{prop}

\begin{proof}

It is easily verified by hand that $H_{6}^{\tau_{1}} = H_{6}$ while $H_{6}^{\tau_{2}}$ is the complex
conjugate $H_{6}^{\ast}$. Therefore both $\tau_1$ and the product $\tau_2\ast$ fix $H_{6}$.
We claim that $\Aut^{\ast}(H_{6}) = \langle \tau_{1}, \tau_{2}\ast\rangle$.

First, we show that the intersection $\Aut^{\ast}(H_{6}) \cap N$ has order $3$. To prove this,
suppose that $(D, E) \in N$, and that $D^{-1}H_{6}E = H_{6}$, or equivalently $DH_{6} = H_{6}E$. Since
the first column of $H_{6}$ is constant, $D$ must be a scalar matrix. So $D$ commutes with $H_{6}$, and
we have $DH_{6} = H_{6}D = H_{6}E$. Hence $D = E$, so $(D, E) = (\ww^{i} I, \ww^{i} I)$ for some $i$.
Since these elements do leave $H_6$ invariant, the claim is proved.

We next claim that there is no element $(D, E)$ of $N$ such that $DH^{\ast}_{6} = H_{6}E$; suppose to the contrary that such a $(D, E)$ exists. Precisely the same argument as before shows that $D$ must be scalar. This implies that $H^{\ast}_{6} = H_{6}ED^{-1}$,
but this equation has no solution in diagonal matrices: since the first row of $H_{6}^{\ast}$ is equal to
the first row of $H_{6}$, we would require $ED^{-1} = I_{6}$, from which we derive $H_{6} = H_{6}^{\ast}$, a contradiction.

Consider the subgroup $K := \langle \tau_{1}, \tau_{2}\ast, N\rangle$ of $X$. Since $X=\langle K, \ast\rangle$ and
$\ast \not\in K$, we have $|X:K|=2$ and $X = K \cup (K\,\ast)$. It follows, moreover, from the previous arguments
that no element of $K$ sends $H_{6}$ to $H_{6}^{\ast}$, and hence no element of the right coset $K\ast$ can fix $H_{6}$.
Therefore, $\Aut^{\ast}(H_{6}) \subseteq K$, and from the first paragraph of the proof we also have $\Aut^{\ast}(H_{6})N = K$.
The quotient $\Aut^{\ast}(H_{6})/(\Aut^{\ast}(H_{6}) \cap N)$ is isomorphic to $K/N$,
an index $2$ subgroup of $X/N\cong \Sy_{6}.2$. In particular $K/N$ contains $A_6$ as a normal subgroup of index 2.
Since the element $N\tau_{2}\ast$ does not lie in $A_6$ and does not centralise $A_6$ it follows that $K/N\cong \Sy_6$.

We have shown that $\Aut^{\ast}(H_{6})$ has a normal subgroup of order $3$ with quotient isomorphic to $\Sy_{6}$.
The elements $(\tau_{2}\ast)^{\tau_{1}^{i}}$ for $0\leq i \leq 4$ project onto a set of Coxeter generators for $\Sy_{6}$.
With these generators, it is straightforward to construct a Sylow $3$-subgroup of $\Aut^{\ast}(H_{6})$. One such subgroup is generated by
\begin{eqnarray*}
 x &:=& \left([\wo, 1, \ww, \ww, 1, \wo] (1,2,3), \quad [\ww, 1, \ww, 1, \wo, \wo] (1, 4, 6)(2, 3, 5) \right)\\
 y &:=&  \left([\ww, \wo, 1, 1, \wo, \ww] (4,5,6),  \quad [\ww, \ww, \ww, \ww, \ww, \ww] (1, 4, 6)(2, 5, 3)\right).
\end{eqnarray*}

A computation shows that $\left[ x,y\right] = \left([\ww, \ww, \ww, \ww, \ww, \ww],  [\ww, \ww, \ww, \ww, \ww, \ww]\right)$. This shows that the commutator subgroup contains the normal subgroup of order $3$, hence the extension is non-split. Elements of $\Aut^{\ast}(H)$ which map onto odd permutations act on $\left[ x,y\right] $ by inversion. So the centraliser of this normal subgroup is of index $2$ in $\Aut^{\ast}(H)$: this is necessarily a non-split central extension $3.A_{6}$.

A perfect group $S$ has a largest non-split central extension $\hat{S}$ which is unique up to isomorphism. The center of $\hat{S}$ is the Schur multiplier of $S$, and every non-split central extension of $S$ is a quotient of $\hat{S}$. The number of generators of the Schur multiplier is bounded by $g - r$ where $g$ is the number of generators in a presentation of $S$ and $r$ is the number of relations. We refer the reader to Wiegold's survey on the Schur multiplier for proofs of all these results \cite{Wiegold}. Since $A_{6}$ is shown in \cite{nicepres} to have the presentation
\[ \langle a, b \mid a^{4}, b^{5}, abab^{-1}abab^{-1}a^{-1}b^{-1} \rangle\,, \]
it follows that the Schur multiplier of $A_{6}$ is cyclic. Hence the non-split extension $3.A_{6}$ is unique up to isomorphism.

Now, since $\Aut^{\ast}(H)$ splits over $3.A_{6}$, we have that $3.A_{6} < \Aut^{\ast}(H) < \Aut(3.A_{6})$.
Suppose that $\xi \in \Aut(3.A_{6})$ such that the image of $\xi$ in $\Aut(A_{6})$ is the trivial automorphism.
Let $\sigma \in 3.A_{6}$ be an element of order $15$, projecting onto a $5$-cycle in $A_{6}$. Then $\sigma^5$
generates the central subgroup of order $3$. Each coset of $\langle \sigma^{5}\rangle$ contains a unique element
of order $5$, which is fixed by hypothesis. So either $\langle \sigma \rangle$ is fixed element-wise, or $\xi = \ast$.
Moreover, any two subgroups of order $15$ intersect in $\langle \sigma^{5}\rangle$, so the action of $\xi$ is
identical on all $5$-cycles. Since the $5$-cycles generate $A_{6}$, the action of $\xi$ is completely determined.

So each choice of actions on $3$ and on $A_{6}$ determines at most one isomorphism class of groups.
It follows that $\Aut^{\ast}(H)$ is uniquely described as the group of shape $3.\Sy_{6}$ with trivial center.


The projection of $\rho_{1}(\Aut^{\ast}(H) \cap X_{0})$ is clearly a faithful linear representation of $3.A_{6}$
over the complex numbers, completing the proof.
\end{proof}

In fact, $3.A_{6}$ is the largest subgroup of $\Aut^{\ast}(H_{6})$ admitting a faithful $6$-dimensional representation over $\mathbb{C}$. So this is $\Aut(H_{6})$. A useful way to understand the actions of $X$ and of $\Aut^{\ast}(H_{6})$ is via a permutation action on $18$ points, which we now describe.
Let $P_1=\tau_1^{\rho_1}$ and $P_2=\tau_2^{\rho_1}$, and define the following $18\times 6$ matrices:
\[\begin{array}{ccc}
M_1=\left(\begin{array}{r}
 H\\
\omega H\\
\omega^2 H
\end{array}\right)&\quad\mbox{and}\quad &M_2=\left(\begin{array}{r}
 H^*\\
\omega H^*\\
\omega^2 H^*
\end{array}\right).\\
\end{array}\]
For $1\leq i\leq 18$, let $\Row_i(M_j)$ denote the $i^{th}$ row of $M_j$ (where $j=1,2$).
Let $P_1$ act on the rows of $M_1$, and similarly the rows of $M_2$, as follows:
\[\begin{array}{ccc}
P_1\cdot M_1&=&\left(\begin{array}{r}
 P_1H\\
\omega P_1H\\
\omega^2 P_1H
\end{array}\right)
\end{array}\]
By letting $P_2$ act on the rows of $M_1$ and $M_2$ in a similar manner, we find that $P_1$ and $P_2$ act in the same way on the rows of $M_1$ and the rows of $M_2$, and hence act on the set $\Omega(18):= \{ \{ \Row_i(M_1), \Row_i(M_2)\} | i=1,\dots,18\}$.  Also, letting * act as complex conjugation on $M_1$ and $M_2$, we see that $\ast$ also induces a permutation of $\Omega(18)$. Thus $\tau_1$, $\tau_2$ and $\ast$ all induce permutations of $\Omega(18)$ and, identifying $\{ \Row_i(M_1), \Row_i(M_2)\}$ with $i$, for each $i$, we get a permutation representation of $X$ on 18 points with the following generating permutations:
\begin{eqnarray*}\tau_1 &=& (2,3,4,5,6)(8,9,10,11,12)(14,15,16,17,18), \\
\tau_2 &=& (1,2)(3,15,9)(4,10,16)(5,11,17)(6,18,12)(7,8)(13,14),\\
  \ast &=& (7, 13)(8, 14)(9, 15)(10, 16)(11, 17)(12, 18).
\end{eqnarray*}
The kernel of $X$ in this action is the subgroup of $N$ of order $3^5$ consisting of pairs with trivial first component.
The restriction to $\Aut^{\ast}(H_{6})$ is faithful, however.
One could construct a faithful action of $X$ by taking the permutation action induced by its action on the rows of $H_{6}$ together with the induced action on columns.

\begin{remark}
The matrix $H_{6}$ and the group $3.A_6$ can be realised over any field $k$ for which $k^{\times}$ has a subgroup of order $3$. In the case that $k$ is the finite field of order $4$, the rows of $H_{6}$ span the \textit{Hexacode}, introduced by Conway as part of a construction for the group $M_{12}$. It is discussed in detail in Section 11.2 of \cite{ConwaySloane}. In particular, this code is the extended quadratic residue code with parameters $(6, 3, 4)$. Uniqueness can easily be verified by hand: observe that the punctured code is the Hamming $(5, 3, 3)$ code, which is unique, and that any pair of one-bit extensions which increase the minimum distance are isomorphic. The $6$-dimensional $\C$-representation of $3\cdot A_{6}$ has been previously described in the literature, normally via its action on a set of vectors in $\mathbb{C}^{6}$ derived from the hexacode. In particular, Wilson gives the action of $3\cdot A_{6}$ on certain vectors of weight $4$ in Section 2.7.4 of \cite{Wilson}.
\end{remark}

\section{The outer automorphism of $\Sy_{6}$}

Finally we construct the outer automorphism of $\Sy_{6}$ over the split-quaternions. Recall that the split-quaternions are a $4$-dimensional $\mathbb{R}$-algebra with basis $[1, i, \beta, \beta i]$ where $[1,i]$ generates the usual algebra of complex numbers and $\beta^{2} = 1$, $i^{\beta} = -i$. We denote the split quaternions by $\mathbb{B}$. They admit an $\mathbb{R}$-linear representation generated by
\[ i \mapsto \left(\begin{array}{rr} 0 & -1 \\ 1 & 0 \end{array} \right), \hspace{1cm} \beta \mapsto \left(\begin{array}{rr} 0 & 1 \\ 1 & 0 \end{array} \right).\]
Observe that $\Aut^{\ast}(H_{6})$ admits a $\mathbb{B}$-linear representation if and only if $\ast$ does, and that the latter is realised by $(\beta I_{6}, \beta I_{6})$.

Since $H_{6}$ is invertible over $\mathbb{C}$, it is invertible over $\mathbb{B}$. Now, rearranging the matrix equation $H_{6}^{\tau_{2}\ast} = H_{6}$, and using the same notation as before for monomial matrices, we obtain that
\[ H_{6} \left[\left[\beta,\beta,\beta\wo,\beta\ww,\beta\ww,\beta\wo \right](1,2)(3,6)(4,5)\right] H_{6}^{-1} = \left[\left[\beta,\beta,\beta\ww,\beta\wo,\beta\wo,\beta\ww \right](1,2)\right].\]
Note that $(\beta\ww)^{2} = (\beta\wo)^{2} = 1$ so that the matrix on the right hand side of the above equation is an involution.

As was the case over the complex numbers, $H_{6}$ intertwines the projections $\rho_{1}$ and $\rho_{2}$. We observe that for any $g \in \Aut^{\ast}(H)$, we have that $g^{\rho_{1}} = H_{6}g^{\rho_{2}}H_{6}^{-1}$. But, as illustrated above, $\tau_{2}^{\rho_{1}\pi}$ is a $2$-cycle, while the projection $\tau_{2}^{\rho_{2}\pi}$ is a product of $3$ disjoint $2$-cycles. We conclude that the representations $\rho_{1}\pi$ and $\rho_{2}\pi$ of $\Sy_{6}$ cannot be conjugate. Thus whereas the permutation representations of $\Sy_{6}$ on $6$ points are not equivalent, and the monomial representations of $3.A_6$ are not equivalent, we have constructed two explicit $\mathbb{B}$-linear representations of $3.\Sy_6$ which are equivalent under conjugation by $H_{6}$. Moreover, although the representation is not defined over $\mathbb{C}$, the intertwiner $H_{6}$ is.

\begin{theorem}
There exists an irreducible $6$-dimensional monomial representation of $3.\Sy_{6}$ over the split-quaternions. Two conjugate representations of $3.\Sy_{6}$ intertwined by the complex Hadamard matrix $H_{6}$ give an explicit construction for the outer automorphism of $\Sy_{6}$.
\end{theorem}

\subsection*{Acknowledgements}
Work on this paper was begun while the second author was visiting the Centre for the Mathematics of Symmetry and Computation at the University of Western Australia in March 2012. The hospitality of the CMSC is gratefully acknowledged,  and in particular support from the ARC Federation Fellowship grant FF0776186 of the third author, which also supported the first author.

The second author acknowledges the support of the Australian Research Council via grant DP120103067, and Monash University where much of this work was completed. This research was partially supported by the Academy of Finland (grants \#276031, \#282938, \#283262 and \#283437). The support from the European Science Foundation under the COST Action IC1104 is also gratefully acknowledged.

\bibliographystyle{abbrv}
\flushleft{
\bibliography{NewBiblio}
}

\end{document}